\documentclass[12pt]{amsart}
\usepackage{accsupp}
\usepackage{amssymb,amscd,amsmath}
\usepackage{amsfonts}
\usepackage{fullpage}
\usepackage{array}
\usepackage{amsthm} 
\usepackage{url}
\usepackage{mathrsfs}
\usepackage{todonotes}
\usepackage{relsize}
\usepackage{enumerate}
\usepackage{comment}
\numberwithin{equation}{section}
\newtheorem{theorem}{Theorem}[section]

\newtheorem{lemma}[theorem]{Lemma}

\theoremstyle{definition}

\newtheorem*{remark*}{Remark}
\newtheorem*{proposition*}{Proposition}
\newtheorem*{acknowledgement*}{Acknowledgement}

\newcommand{\Z}{\mathbb{Z}}
\newcommand{\Q}{\mathbb{Q}}
\newcommand{\R}{\mathbb{R}}
\newcommand{\C}{\mathbb{C}}

\DeclareMathOperator{\GL}{GL}

\newcommand{\SL}{SL}

\newcommand{\hyp}[4]{~_2F_1\left(\left.\begin{smallmatrix}#1,~#2\\#3\end{smallmatrix}\right\rvert#4\right)}

\renewcommand{\Re}{\text{Re\,}}
\DeclareMathOperator{\Res}{Res}

\newcommand{\D}{\mathcal{D}}

\newcommand{\uhp}{\mathcal{H}}

\newcommand{\im}{\textrm{Im}}
\newcommand{\re}{\textrm{Re}}

\newcommand{\eps}{\varepsilon}

\title{Twisting moduli for $GL(2)$}
\author[B. Bedert]{Benjamin Bedert}
\author[G. Cooper]{George Cooper$^\star$}
\author[T. Oliver]{Thomas Oliver}
\author[P. Zhang]{Pengcheng Zhang}
\thanks{$^{\star}$This work was partially supported by the EPSRC vacation bursary and the Andrew Mason Memorial Scholarship (Balliol College).}
\address{(B.B., P.Z.) St John’s College, Oxford, OX1 3JP, UK.}
\email{benjamin.bedert@sjc.ox.ac.uk, pengcheng.zhang@sjc.ox.ac.uk}
\address{(G.C.) Balliol College, Oxford, OX1 3BJ, UK.}
\email{george.cooper@balliol.ox.ac.uk}
\address{(T.O.) Mathematical Institute, Andrew Wiles Building, University of Oxford, Radcliffe Observatory Quater, Woodstock Road, Oxford, OX2 6GG, UK.}
\email{thomas.oliver@maths.ox.ac.uk}
\date{\today}
\usepackage{hyperref}
\hypersetup{pdfauthor={Thomas Oliver},%
	            pdftitle={Twisting Moduli},%
	            pdfsubject={Number Theory},%
	            pdfcreator={XeLaTeX},%
	            colorlinks=true,%
	            linkcolor=[rgb]{0,.6,1},
	            citecolor=blue}
\begin{document}
\maketitle
\subsection*{Abstract}
We prove various converse theorems for automorphic forms on $\Gamma_0(N)$, each assuming fewer twisted functional equations than the last. We show that no twisting at all is needed for holomorphic modular forms in the case that $N\in\{18,20,24\}$ - these integers are the smallest multiples of $4$ or $9$ not covered by earlier work of Conrey--Farmer. This development is a consequence of finding generating sets for $\Gamma_0(N)$ such that each generator can be written as a product of special matrices. As for real-analytic Maass forms of even (resp. odd) weight we prove the analogous statement for $1\leq N\leq 12$ and $N\in\{16,18\}$ (resp. $1\leq N\leq 12$, $14\leq N\leq18$ and $N\in\{20,23,24\}$).
\section{Introduction}\label{section.Intro}
The use of twisted functional equations in the characterisation of automorphic representations dates back to Weil's archetypal converse theorem for holomorphic modular forms \cite{UDBDRDF}. It is a problem of long-standing interest to limit the number of required twists \cite{PSconverse}, \cite{modFormsNDirichletSeries}, \cite{Wli}, \cite{EOHCT}, \cite{DPZ}. In \cite{PSconverse} it is shown that, for any fixed prime $p$, it suffices to assume the analytic properties of twists by primitive characters modulo $p^r$ for all $r\geq0$. On the other hand, the main result of \cite{DPZ} states that there exists a prime $q$ such that the analytic properties of twists by primitive characters modulo $q$ are sufficient to capture the holomorphic modular forms on $\Gamma_0(N)$. In \cite{CEOHCT} it is demonstrated that there is a density 1 subset of the primes from which we can choose $q$.  Still in the holomorphic case it was shown in \cite{EOHCT} and \cite{CFOS} that, for $1\leq N\leq17$ and $N=23$, no twisting at all is required. This result generalises earlier work of Hecke which applies to the cases $1\leq N\leq4$ \cite{Hecke1936}. 

In Theorem~\ref{thm:WeilMaass} we prove a converse theorem for Maass forms on $\Gamma_0(N)$ assuming twisted functional equations for primitive Dirichlet characters with the same moduli as Weil \cite{UDBDRDF}. For context, note that a converse theorem for Maass forms of small level was established by Maass \cite{Maass1949} and the converse theorem of Jacquet--Langlands \cite{AFOGL2} applies to Maass forms on $\GL_2(\mathbb{A}_F)$. Nevertheless, generalising Weil's classical approach to real-analytic forms had been regarded as a difficult problem, see for example \cite[Section~3.4]{GM}, with two different approaches having appeared only recently in \cite[Section~3]{TCTMF} and \cite{MSSU}. We note that the results of \cite{MSSU} apply to half-integer weight, though they make the additional assumption of analytic properties for twists by imprimitive characters. Theorem~\ref{thm:WeilMaass} uses an extension of the method used in \cite[Section~3]{TCTMF} to Maass forms of arbitrary integer weight $k$. In particular, we develop further the method of ``two circles'' implemented in \cite{TCTMF}. By the theory of weight raising and lowering operators, it is enough to work with the cases $k=0,1$. As \cite[Theorem~3.1]{TCTMF} solves the problem when $k=0$, it remains to solve the case $k=1$. In Theorem~\ref{thm.DPZmaass} we show that it is in fact sufficient to assume only twists by primitive Dirichlet characters modulo a single prime $q$ constrained by congruence conditions as in \cite{DPZ} (equation~\eqref{eq.specialmodulus}). In Theorem~\ref{thm.CFMaass}, we show that no twisting at all is required for $1\leq N\leq 12$ and $N\in\{16,18\}$ when $k=0$ and $1\leq N\leq 12$, $14\leq N\leq18$ and $N\in\{20,23,24\}$ when $k=1$. In Theorem~\ref{thm.Holomorphic}, we deduce from the same argument that the main result of \cite{EOHCT} can be extended to the cases $N\in\{18,20,24\}$. The numbers 20 and 24 (resp. 18) are the smallest multiples of $2^2$ (resp. $3^2$) not covered by Conrey--Farmer. Our proof works by writing generating sets for $\Gamma_0(N)$ such that each generator can be written as a product of special matrices. Whilst the matrices in $\Gamma_0(N)$ have integer coefficients, we sometimes allow matrices with non-integer coefficients amongst the factors. In each of the Theorems established here, it would be interesting to allow for the various $L$-functions to have poles as in \cite{WCTWP}, \cite[Theorem~1.1]{TCTMF}, and \cite[Theorem~1.2]{RALF}.
\subsection*{Notation}
By $\mathcal{H}$ we denote the upper half-plane, the smooth functions on which are denoted by $C^{\infty}(\uhp)$. The weight-$k$ Laplace-Beltrami operator on $C^{\infty}(\uhp)$ is
\begin{equation}\label{eq.WeightkLBoperator}
\Delta_k = -y^2 \left( \frac{\partial^2}{\partial x^2} + \frac{\partial^2}{\partial y^2} \right) + iky \frac{\partial}{\partial x}. 
\end{equation}
A weight-$k$ Maass form on $\Gamma_0(N)$ is a real-analytic eigenfunction of $\Delta_k$ with polynomial growth at the cusps of $\Gamma_0(N)$ satisfying a weight-$k$ automorphic transformation law with respect to $\Gamma_0(N)$. Given two complex numbers $\kappa$ and $\mu$, the Whittaker function $W_{\kappa,\nu}:\mathbb{R}_{>0}\rightarrow\mathbb{C}$ is annihilated by the differential operator:
\begin{equation}\label{eq.WhittakerODE}
\frac{d^2}{dy^2} + \left( -\frac{1}{4} + \frac{\kappa}{y} + \frac{1/4 - \mu^2}{y^2} \right), 
\end{equation}
and satisfies the following asymptotic formula as $y \rightarrow \infty$:
\begin{equation}\label{eq.WhittakerGrowth}
W_{\kappa,\mu}(y) \sim e^{-\frac{y}{2}} y^{\kappa}.
\end{equation}
We use the following variants of the gamma function:
\begin{equation}
\Gamma_{\R}(s)=\pi^{-s/2}\Gamma\left(\frac{s}{2}\right),\ \ \Gamma_{\C}(s)=(2\pi)^{-s}\Gamma(s).
\end{equation}
If $\psi$ is a Dirichlet character mod $q$, then the associated Gauss sum is:
\begin{equation}
\tau(\psi) = \sum_{a \text{ mod } q} \psi(a) \exp \left( 2\pi ia/q \right).
\end{equation}

\subsection*{Acknowledgements}
The first and fourth authors would like to thank Yicheng Sun for his help with computer calculations used in Section~\ref{sec.CFmaass}. The second author was supported by an EPSRC vacation bursary and Andrew Mason Memorial Scholarship (Balliol College). The third author thanks Dan Fretwell and Michalis Neururer.

\section{Theorems}
In the Theorems below, a function $F(s)$ on $\C$ is called \textsl{EBV} if it is entire and bounded on every vertical strip.
\begin{theorem} \label{thm:WeilMaass}
Let $\varepsilon\in\{\pm 1\}$ and $\nu\in\mathbb{C}\backslash\{0\}$. Say $N\in\Z_{>0}$ and let $\chi$ be a Dirichlet character mod $N$. Assume that $(a_n)_{n=-\infty}^{\infty},(b_n)_{n=-\infty}^{\infty}$ are complex sequences such that $a_n,b_n=O(n^{\sigma})$ for some $\sigma\in\mathbb{R}_{>0}$ and, for $n > 0$, $a_{-n} = \eps \nu a_n$ and $b_{-n} = \eps \nu b_n$.  Given a Dirichlet character $\psi$ modulo $p$ (either 1 or a prime number), define:
\begin{equation}\label{eq.L}
L_f(s, \psi) = \sum_{n = 1}^{\infty} \psi(n) a_n n^{-s}, \ \ L_g(s, \psi) = \sum_{n = 1}^{\infty} \overline{\psi}(n) b_n n^{-s}, \ \ \Re(s)>\sigma+1.
\end{equation}
and 
\begin{equation}\label{eq: completed l functions}
\begin{split}
\Lambda_f(s, \psi) = \Gamma_{\R} \left( s + \frac{1 + \psi(-1) \eps}{2} + \nu \right) \Gamma_{\R} \left( s + \frac{1 - \psi(-1) \eps}{2} - \nu \right) L_f(s, \psi), \\ 
\Lambda_g(s, \psi) = \Gamma_{\R} \left( s + \frac{1 + \psi(-1) \eps}{2} + \nu \right) \Gamma_{\R} \left( s + \frac{1 - \psi(-1) \eps}{2} - \nu \right) L_g(s, \psi).
\end{split}
\end{equation}
If $\psi=\textbf{1}$ is the trivial Dirichlet character we omit it from the notation. For a finite set $\mathcal{S}$ of primes including $2$ and those dividing $N$, define $\mathcal{P}$ to be the complement of $\mathcal{S}$ in the set of all primes. Assume the following:
\begin{enumerate}
\item The functions $\Lambda_f(s)$ and $\Lambda_g(s)$ continue to holomorphic functions on $\C-\{\varepsilon\nu,1-\varepsilon\nu\}$ with at most simple poles in the set $\{ \eps \nu, 1 - \eps \nu \}$, are uniformly bounded on every vertical strip outside of a small neighbourhood around each pole and satisfy the functional equation
\begin{equation}\label{eq.FE1}
\Lambda_f(s) = N^{\frac{1}{2} - s} \Lambda_g(1 - s).
\end{equation}
\item For all primitive Dirichlet characters $\psi$ of conductor $q \in \mathcal{P}$, continue to entire functions which are EBV and satisfy the functional equation
\begin{equation}\label{eq.TwistedFE} 
\Lambda_f(s, \psi) = \psi(N) \chi(q) \frac{\tau(\psi)}{\tau(\overline{\psi})} (q^2N)^{\frac{1}{2} - s} \Lambda_g(1 - s, \overline{\psi}).
\end{equation}
\end{enumerate}
Define $f_0, \ g_0 : \R_{>0} \rightarrow \C$ by
\begin{equation}\label{eq.f0}
f_0(y) = - \Res_{s = \eps \nu} \ \Lambda_f(s) y^{\frac{1}{2} - \eps \nu}, \quad g_0(y) = \Res_{s = \eps \nu} \ \Lambda_g(s) y^{\frac{1}{2} - \eps \nu},
\end{equation}
define $\tilde{f}, \ \tilde{g} : \uhp \rightarrow \C$ by
\begin{equation}\label{eq.ftilde}
\begin{split}
\tilde{f}(x + iy) = \sum_{n=1}^{\infty} \frac{a_n}{\sqrt{\pi n}} \left( W_{\frac{1}{2}, \nu}(4 \pi ny) e(nx) + \eps \nu W_{-\frac{1}{2}, \nu}(4 \pi ny) e(-nx) \right), \\
\tilde{g}(x + iy) = \sum_{n=1}^{\infty} \frac{b_n}{\sqrt{\pi n}} \left( W_{\frac{1}{2}, \nu}(4 \pi ny) e(nx) + \eps \nu W_{-\frac{1}{2}, \nu}(4 \pi ny) e(-nx) \right),
\end{split}
\end{equation}
and define $f,g:\mathcal{H}\rightarrow\C$ by
\begin{equation}\label{eq.deffg}
f(x+iy)=f_0(y)+\tilde{f}(x+iy),\ \ g(x+iy)=g_0(y)+\tilde{g}(x+iy).
\end{equation}
Then $f$ and $g$ are weight-$1$ Maass forms on $\Gamma_0(N)$ satisfying
\begin{equation}\label{eq.gzf1z}
f(z) = \frac{iz}{|z|}\ \ g\left( -\frac{1}{Nz} \right).
\end{equation}
\end{theorem} 
In the above Theorem, we do not assume that the $L$-functions have an Euler product. Including this assumption, we may restrict the set of twisting moduli. 
\begin{theorem}\label{thm.DPZmaass}
Let $\varepsilon$, $\nu$, $N$, $\chi$, $(a_n)_{n=-\infty}^{\infty}$, and $(b_n)_{n=-\infty}^{\infty}$ be as in Theorem~\ref{thm:WeilMaass} and let $k\in\{0,1\}$. Given a Dirichlet character $\psi$ mod $p$ (either 1 or a prime) define $L_f(s,\psi),L_g(s,\overline{\psi})$  as in equation~\eqref{eq.L} and define
\begin{equation}\label{eq: completed l functions}
\begin{split}
\Lambda_f(s, \psi) = \Gamma_{\R} \left( s + \frac{1 - (-1)^k \psi(-1) \eps}{2} + \nu \right) \Gamma_{\R} \left( s + \frac{1 - \psi(-1) \eps}{2} - \nu \right) L_f(s, \psi), \\ 
\Lambda_g(s, \psi) = \Gamma_{\R} \left( s + \frac{1 - (-1)^k \psi(-1) \eps}{2} + \nu \right) \Gamma_{\R} \left( s + \frac{1 - \psi(-1) \eps}{2} - \nu \right) L_g(s, \psi). 
\end{split}
\end{equation}
Assume the following:
\begin{enumerate}
\item For $\Re(s)>\sigma+1$, the functions $L_f(s)$ and $L_g(s)$ have an Euler product expansion of the form 
\begin{equation}\label{eq.EPfg} 
\begin{split}
L_f(s) = \prod_{p \nmid N} \left(1 - a_p p^{-s} + p^{-2s}\right)^{-1}\prod_{\substack{p \mid N\\p^2\nmid N}} \left( 1 - p^{-s} \right)^{-1} ,\\
L_g(s) = \prod_{p \nmid N} \left(1 - b_p p^{-s} + p^{-2s}\right)^{-1}\prod_{\substack{p \mid N\\ p^2\nmid N}} \left( 1 - p^{-s} \right)^{-1} .
\end{split}
\end{equation} 
\item The functions $\Lambda_f(s)$ and $\Lambda_g(s)$ admit extensions to all of $\C$ which are EBV and satisfy the functional equation~\eqref{eq.FE1}.
\item Given a generating set $\left\{\begin{pmatrix}A_j&B_j\\C_jN&D_j\end{pmatrix}\in\Gamma_1(N):j=1,\dots,h\right\}$ for $\Gamma_1(N)$, there is a prime $q$ satisfying
\begin{equation}\label{eq.specialmodulus}
q\equiv A_j~(\text{mod }q|C_j|),~j=1,\dots,h,
\end{equation} 
such that for all primitive Dirichlet characters $\psi$ mod $q$ the functions $\Lambda_f(s, \psi)$ and $\Lambda_g(s, \overline{\psi})$ admit extensions which are EBV and satisfy the functional equation:
\begin{equation}\label{eq.FEtwist}
\Lambda_f(s, \psi) = \varepsilon^{1-k}\psi(N) \chi(q) \frac{\tau(\psi)}{\tau(\overline{\psi})} (q^2N)^{\frac{1}{2} - s} \Lambda_g(1 - s, \overline{\psi}).
\end{equation}
\end{enumerate}
Then $\Lambda_f(s,\psi)$ and $\Lambda_g(s,\overline{\psi})$ satisfy functional equation~\eqref{eq.TwistedFE} for all primitive Dirichlet characters $\psi$ mod $p\in\mathcal{P}$.
\end{theorem}
An upper bound for the smallest prime $q$ satisfying equation~\eqref{eq.specialmodulus} is given in \cite[Section~3]{DPZ}. 
\begin{theorem}\label{thm.CFMaass}
Let $a_n$, $\sigma$ be as in Theorem~\ref{thm:WeilMaass}, and let $\nu\in\mathbb{C}$. Given $k=0$ (resp. $k=1$), let $N\in\mathbb{Z}_{>0}$ satisfy $1\leq N\leq 12$ and $N\in\{16,18\}$ (resp. $1\leq N\leq 12$, $14\leq N\leq18$ and $N\in\{20,23,24\}$), and let $\chi$ be a Dirichlet character mod $N$. For $\Re(s)>\sigma+1$, define $L_f(s)$ as in equation~\eqref{eq.L} and suppose that $L_f(s)$ has an Euler product as in equation~\eqref{eq.EPfg}. Define $\Lambda_f(s)$ as in equation \eqref{eq: completed l functions}, and suppose that $\Lambda_f(s)$ admits an extension to $\mathbb{C}$ which is EBV and satisfies the functional equation
\begin{equation}
\Lambda_f(s)=N^{\frac12-s}\Lambda_f(1-s). 
\end{equation}
Then $f$ is a weight-$k$ Maass form for $\Gamma_0(N)$.
\end{theorem}
\begin{theorem}\label{thm.Holomorphic}
Let $a_n$, $\sigma$ be as in Theorem~\ref{thm:WeilMaass}, let $k\in\mathbb{Z}_{>0}$ be an even integer, and $N\in\{18,20,24\}$. For $\Re(s)>\sigma+1$, define $L_f(s)$ by equation~\eqref{eq.L}. Suppose that $L_f(s)$ has an Euler product expansion 
\begin{equation}\label{eq.EPf} 
L_f(s) = \prod_{p \nmid N} \left(1 - a_p p^{-s} + p^{k+1-2s}\right)^{-1}\prod_{\substack{p \mid N\\p^2\nmid N}} \left( 1 - p^{k/2-1-s} \right)^{-1}.
\end{equation} 
Define
\begin{equation}\label{eq.holcompletions}
\Lambda_f(s)=\Gamma_{\C}(s)L_f(s),
\end{equation}
and suppose that $\Lambda_f(s)$ admits an extensions to $\C$ which is EBV and satisfies the functional equation:
\begin{equation}\label{eq.holFE}
\Lambda_f(s)=(-1)^{k/2}N^{\frac{k}{2}-s}\Lambda_f(k-s)
\end{equation}
If,
\begin{equation}\label{eq.holexp}
f(z)=\sum_{n=1}^{\infty}a_n\exp(2\pi inz),
\end{equation}
then $f$ is a weight-$k$ modular form on $\Gamma_0(N)$.
\end{theorem}
The conclusion of Theorem~\ref{thm.Holomorphic} is weaker than that in \cite{EOHCT}, in which $f$ is shown the be cuspidal. In both Theorems~\ref{thm.CFMaass} and~\ref{thm.Holomorphic}, cuspidality of $f$ would be a consequence of the convergence of $L_f(s)$ in $\Re(s)>1-\delta$ for any $\delta>0$. For a given $N$, this assumption may or may not be necessary (cf. \cite[Section~5]{EOHCT}). 
\section{Proofs}
\subsection{Weil's Lemma}
Let $\GL_2^+(\R)$ denote the subgroup of $\GL_2(\mathbb{R})$ consisting of matrices with positive determinant, which contains $\SL_2(\R)$ as a subgroup. For $k\in\Z_{\geq0}$, the weight-$k$ action of $\gamma = \begin{pmatrix}a & b \\ c & d \end{pmatrix}\in\text{GL}_2^+(\R)$ on $f:\uhp\rightarrow\C$ is given by
\begin{equation}\label{eq.weightkaction}
(f |_k \gamma)(z) = \exp\left(-ik\arg(cz+d)\right)f \left( \frac{az + b}{cz + d} \right).
\end{equation}
For $r \in \Q$ write 
\begin{equation}
P_r = 
\begin{pmatrix} 
1 & r \\ 0 & 1 
\end{pmatrix}
\in\mathrm{SL}_2(\R).
\end{equation} 
For $N\in\mathbb{Z}_{>0}$ set
\begin{equation}\label{eq.FrickeMatrix}
H_N = \begin{pmatrix} 0 & -1 \\ N & 0 \end{pmatrix}\in\GL_2^+(\R).
\end{equation} 
Equation~\eqref{eq.FE1} is equivalent to
\begin{equation}\label{eq.FEslash}
g = i^k \ f |_k H_N.
\end{equation}
Extending~\eqref{eq.weightkaction} by linearity, we equip $C^{\infty}(\uhp)$ with a right $\C[\text{GL}_2^+(\R)]$-module structure. For $f\in C^{\infty}(\uhp)$, we introduce the following right-ideal of $\C [\text{GL}_2^+(\R)]$:
\begin{equation}\label{eq.AnnIdeal}
\Omega_{f} = \{W \in \C [\text{GL}_2^+(\R)] : f|_k W = 0\}.
\end{equation}
For elements $\gamma_1, \gamma_2$ of $\C[\text{GL}_2^+(\R)]$ we will write $\gamma_1 \equiv \gamma_2$ to mean $\gamma_1 - \gamma_2 \in \Omega_f$. If $f$ has an expansion as in equation~\eqref{eq.deffg}, then
\begin{equation}\label{eq.periodic}
1-P_r\equiv0,\ \ r\in\mathbb{Z}.
\end{equation}
For a prime number $p$, recall the Hecke operator
\begin{equation}
T_p = \frac{1}{\sqrt{p}}\left(\begin{pmatrix}p&0\\0&1\end{pmatrix}+\sum_{a = 0}^{p-1} \begin{pmatrix} 1 & a \\ 0 & p \end{pmatrix}\right)\in\C[\GL^+_2(\R)],
\end{equation}
and the operator
\begin{equation}
U_p=\sum_{a = 0}^{p-1} \begin{pmatrix} p & a \\ 0 & p \end{pmatrix}\in\C[\GL^+_2(\R)].
\end{equation} 
If $L_f(s)$ has an Euler product as in equation~\eqref{eq.EPfg}, then, for $p\nmid N$,
\begin{equation}\label{eq.apTp}
a_p-T_p\equiv0.
\end{equation}
On the other hand, if $p|N$, then
\begin{equation}
U_p\equiv\begin{cases} \begin{pmatrix} p& 0 \\ 0&1 \end{pmatrix}, & p^2\nmid N, \\ 0, & p^2|N. \end{cases}
\end{equation}
We say a matrix $E\in\SL_2(\mathbb{R})$ is elliptic if it has a unique fixed point in $\uhp$. 
\begin{lemma}\label{prop:2circles}
Let $k=0$ and let $c:\uhp\rightarrow\C$ be a continuous function. Say there exist $E_1,E_2\in\mathrm{SL}_2(\R)$ such that $E_1$ is an elliptic matrix of infinite order with fixed point $a\in\uhp$, $a$ is not a fixed point of $E_2$, and $c|_0E_1=c|_0E_2=c$, then $c$ is constant on $\mathcal{H}$. 
\end{lemma}
\begin{proof}
Since $E_1$ is elliptic, we know that $|\mathrm{tr}(E_1)|=|\mathrm{tr}(E_2E_1E_2^{-1})|<2$. Subsequently, we deduce that $E_3=E_2E_1E_2^{-1}$ is an elliptic matrix. Since $E_1$ has infinite order, so does $E_3$. Moreover, $E_3$ fixes $E_2a\neq a$ by assumption. Therefore $c$ is preserved by the two infinite order elliptic matrices with distinct fixed points and hence the result follows from \cite[Theorem~3.10]{TCTMF}.
\end{proof}
\begin{lemma} \label{prop: one circle}
Let $k=1$ and let $c: \uhp \rightarrow \C$ be a continuous function. If there exists an elliptic matrix $E$ of infinite order such that $c|_1E = c$, then $c = 0$. 
\end{lemma}
\begin{proof}
Let $a\in\mathcal{H}$ denote the unique fixed point of $E$. The elliptic matrix $E$ is diagonalised by the Cayley transform:
\[
K = \frac{1}{\sqrt{a - \overline{a}}} \begin{pmatrix} 1 & -a \\ 1 & -\overline{a} \end{pmatrix}\in\GL_2(\mathbb{C}), 
\]
which takes $\mathcal{H}$ to the open unit disc $\mathcal{D}$ and maps the point $a$ to $0$. Because $E$ is of infinite order, we have
\[
KEK ^{-1}=\begin{pmatrix} e^{i \pi \theta} & 0 \\ 0 & e^{-i \pi \theta} \end{pmatrix},\ \ \theta\in\mathbb{R}\backslash\mathbb{Q}.
\]
Although $K$ is not a matrix in $\GL_2^+(\R)$, we define the function $\tilde{c}=\left(c|_1K\right):\mathcal{D}\rightarrow\mathbb{C}$ by equation~\eqref{eq.weightkaction} with $k=1$ and $\gamma=K$. For all $z \in \D$ and $m \in \Z$ we have 
\[
\tilde{c}(z) = e^{\pi i m \theta} \tilde{c}(e^{2 \pi i m \theta} z).
\] 
Denote by $S^1$ the unit circle. Fix $z \in \D$ and consider the continuous function 
\[S^1 \rightarrow \C\]
\[\omega \mapsto  \omega \tilde{c}(\omega^2 z).\] 
Given two topological spaces $X$ and $Y$, if $Y$ is Hausdorff and $c_1, \ c_2 : X \rightarrow Y$ are continuous functions which agree on some dense subset of $X$ then $c_1 = c_2$. In particular, because $\theta$ is irrational, we deduce the following for all $\omega\in S^1$:
\[
\tilde{c}(z) = \omega \tilde{c}(\omega^2 z).
\]
Taking $\omega = -1$ gives $\tilde{c}(z) = - \tilde{c}(z)$ for all $z \in \D$, so $\tilde{c} = 0$ and therefore $c = 0$. 
\end{proof}
\subsection{Proof of Theorem~\ref{thm:WeilMaass}}\label{WMF}
In this Section we follow certain conventions from \cite{BCK}. Let  $\sigma,\varepsilon,\nu,a_n,b_n$ be as in Theorem~\ref{thm:WeilMaass}. The Fourier expansions of $f$ and $g$ immediately imply that they are eigenfunctions of $\Delta_1$.  Given $\mathcal{P}$ be as in Theorem~\ref{thm:WeilMaass}, let $q \in \mathcal{P}$ and $a\in\mathbb{Z}$ satisfy $(a, q) = 1$ or $a=0$. Denote by $\alpha$ the quotient $a/q$. For $m\in\mathbb{Z}_{\geq0}$, define:
\begin{equation}\label{eq.additivetwists}
L_f\left(s, \alpha, \cos^{(m)}\right) = \sum_{n = 1}^{\infty} \cos^{(m)}\left(2 \pi n \alpha\right) a_n n^{-s},
\end{equation}
where $\cos^{(m)}$ the $m$-th derivative of the cosine function. We have:
\begin{multline}\label{eq: trig identities}
    L_f\left(s, \alpha, \cos^{(m)}\right) = \frac{i^m}{q - 1} \sum_{\substack{\psi~\mathrm{mod}~q\\ \psi(-1) = (-1)^m}} \tau(\overline{\psi}) \psi(a) L_f(s, \psi) \\
    + \begin{cases} (-1)^{\frac{k}{2}} \left( L_f(s) - \frac{q}{q - 1} L_f(s, \mathbf{1}_q) \right), & m \text{ even}, \\ 
0, & m\text{ odd}, \end{cases}
\end{multline}
where the summation on the right-hand side is over the set of all primitive Dirichlet characters of conductor $q$, and $\mathbf{1}_q$ denotes the trivial character modulo $q$. For $m\in\mathbb{Z}_{>0}$ we let $\langle m \rangle$ denote $+$ if $m$ is even and $-$ if $m$ is odd, and define:
\begin{equation}\label{eq.completedadditivetwists}
\Lambda_f\left(s, \alpha, \cos^{(m)}\right) = \gamma_f^{\langle m \rangle}(s) L_f\left(s, \alpha, \cos^{(m)}\right),
\end{equation} 
where 
\begin{equation}\label{eq.gammapm}
\gamma_f^{\pm}(s) =\Gamma_{\R} \left( s + \nu + \frac{1 \mp(-1)^k \varepsilon}{2} \right) \Gamma_{\R} \left( s - \nu + \frac{1 \mp \varepsilon}{2} \right).
\end{equation} 
Equation~\eqref{eq: trig identities} implies:

\begin{multline}\label{eq.completedadditivetwist}
    \Lambda_f\left(s, \alpha, \cos^{(m)}\right) = \frac{i^m}{q - 1} \sum_{\substack{\psi~\mathrm{mod}~q \\ \psi(-1) = (-1)^k}} \tau(\overline{\psi}) \psi(a) \Lambda_f(s, \psi) \\
    + \begin{cases} (-1)^{\frac{m}{2}} \left( \Lambda_f(s) - \frac{q}{q - 1} \Lambda_f(s, \mathbf{1}_q) \right), & \text{ if }m \text{ is even}, \\ 0, & \text{ if } m \text{ is odd}. \end{cases}
\end{multline}
Applying equation~\eqref{eq.TwistedFE}, we deduce:
\begin{multline}\label{eq.FEaddtwist} 
\Lambda_f\left(s, \alpha, \cos^{(m)}\right) = \frac{(-i)^m \chi(q) (q^2N)^{\frac{1}{2} - s}}{q - 1} \sum_{\substack{\psi~\mathrm{mod}~q \\ \psi(-1) = (-1)^m}} \tau(\psi) \psi(Na) \Lambda_g(1 - s, \overline{\psi}) \\
+ \begin{cases} (-1)^{\frac{m}{2}} \left( \Lambda_f(s) - \frac{q}{q - 1} \Lambda_f(s, \mathbf{1}_q) \right), & m\text{ even}, \\ 0, & m\text{ odd}. \end{cases}
\end{multline}
It follows from \cite[(6.699.2),~(6.699.3),~(9.234.1),~(9.234.2),~(9.235.2)]{integrals} that, for $\ell\in\{0,1\}$, $w\in\R$, and $\re(s)\gg0$: 
\begin{multline}\label{eq.Mellinh}
\int_0^{\infty} f(wy + iy + \alpha) y^{s - \frac{1}{2}} \ \frac{dy}{y} \\
= \sum_{n = 1}^{\infty} \frac{a_n}{n^s} \sum_{\ell \in \{ 0, 1 \}} (-i)^{\ell} \cos^{(\ell)}(2 \pi n \alpha) \left( \gamma_f^{\langle \ell \rangle}(s) \ \hyp{ \frac{s + \nu + (1 + (-1)^{\ell} \varepsilon)/2}{2} }{ \frac{s - \nu + (1 - (-1)^{\ell} \varepsilon)/2}{2} }{ \frac{1}{2} }{-w^2}\right. \\
\left.+ 2 \pi i w \gamma_f^{\langle \ell + 1 \rangle}(s + 1) \ \hyp{ \frac{s + \nu + (3 + (-1)^{\ell + 1} \varepsilon)/2}{2} }{, \frac{s - \nu + (3 - (-1)^{\ell + 1} \varepsilon)/2}{2} }{ \frac{3}{2} }{-w^2} \right) .
\end{multline}
Defining, for $w\in\R$ and $\ell \in \{0, 1 \}$,
\begin{multline}\label{eq.Hfl}
    H_f^{\ell}(s, w) = (-i)^{\ell} \left( \ \hyp{ \frac{s + \nu + (1 + (-1)^{\ell} \varepsilon)/2}{2} }{ \frac{s - \nu + (1 - (-1)^{\ell} \varepsilon)/2}{2} }{ \frac{1}{2} }{-w^2}\right. \\
    \left.+ 2 \pi i w \frac{\gamma_f^{\langle \ell + 1 \rangle}(s + 1)}{\gamma_f^{\langle \ell \rangle}(s)} \ \hyp{ \frac{s + \nu + (3 + (-1)^{\ell + 1} \varepsilon)/2}{2} }{ \frac{s - \nu + (3 - (-1)^{\ell + 1} \varepsilon)/2}{2} }{ \frac{3}{2} }{-w^2} \right),
\end{multline}
we may rewrite equation~\eqref{eq.Mellinh} as 
\begin{equation}\label{eq.Mellinh2}
\int_0^{\infty} f(wy + iy + \alpha) y^{s - \frac{1}{2}} \ \frac{dy}{y} = \sum_{\ell \in \{0, 1\}} H_f^{\ell}(s, w) \Lambda_f\left(s, \alpha, \cos^{(\ell)}\right). 
\end{equation}
Applying Mellin inversion to equation~\eqref{eq.Mellinh2} with $\alpha = 0$, we obtain for $c > 1 + |\re(\nu)|$:
\begin{equation}\label{eq.IMtf}
\tilde{f}(wy + iy)  = \frac{1}{2 \pi i} \int_{(c)} \Lambda_f(s) H^0_f(s, w) y^{\frac{1}{2} - s} \ ds, 
\end{equation}
where $(c)$ is the vertical line $\Re(s)=c$. Note that $H^0_f(s,w)$ is an entire function of $s$, and so the only possible singularities of the integrand in equation~\eqref{eq.IMtf} are simple poles at the points $\varepsilon \nu$ and $1 - \varepsilon \nu$. As a consequence of \cite[Lemma~4.1]{TCTMF}, the integrand in \eqref{eq.IMtf} decays to zero in vertical strips. Shifting the line of integration to $(1-c)$ gives by Cauchy's residue theorem:
\begin{multline} \label{eq: f tilde}
\frac{1}{2 \pi i} \int_{(1 - c)} \Lambda_f(s) H^0_f(s, w) y^{\frac{1}{2} - s} \ ds + \sum_{p \in \{ \varepsilon \nu, 1 - \varepsilon \nu \} } \Res_{s = p} \ \Lambda_f(s) H^0_f(s, w) y^{\frac{1}{2} - s} \\
= \frac{1}{2 \pi i} \int_{(c)} \Lambda_f(1 - s) H^0_f(1 - s, w) y^{s - \frac{1}{2}} \ ds + \sum_{p \in \{ \varepsilon \nu, 1 - \varepsilon \nu \} } \Res_{s = p} \ \Lambda_f(s) H^0_f(s, w) y^{\frac{1}{2} - s}.
\end{multline}
Using the identities
\begin{equation}\label{eq.Hfidentity}
H^0_f(1 - s, w) = i \left( \frac{|w + i|}{w + i} \right) (1 + w^2)^{s - \frac{1}{2}} H^0_f(s, -w), \ \ H^0_f = H^0_g,
\end{equation}
and applying equation \eqref{eq.TwistedFE}, we compute 
\begin{multline}
    \tilde{f}(wy + iy) - R(z)     = i \left( \frac{|w + i|}{w + i} \right) \cdot \frac{1}{2 \pi i} \int_{(c)} (N(1 + w^2)y)^{s - \frac{1}{2}} \Lambda_g(s) H^0_g(s, -w) \ ds \\   
    = i \left( \frac{|w + i|}{w + i} \right) \tilde{g} \left( -\frac{w}{N(w^2 + 1)y} + \frac{i}{N(w^2 + 1)y} \right)
   = i \left( \frac{|w + i|}{w + i} \right) \tilde{g} \left( -\frac{1}{N(wy + iy)} \right),
\end{multline}
where 
\begin{multline}
R(z) = \sum_{p \in \{ \varepsilon \nu, 1 - \varepsilon \nu \} } \Res_{s = p} \ \Lambda_f(s) H^0_f(s, w) y^{\frac{1}{2} - s}\\
= H^0_f(\varepsilon \nu, w) \Res_{s = \varepsilon \nu} \ \Lambda_f(s) y^{\frac{1}{2} - \varepsilon \nu}  
+ i \left( \frac{|w + i|}{w + i} \right) (N(1 + w^2)y)^{- \frac{1}{2} + \varepsilon \nu} H^0_g(\varepsilon \nu, -w) \Res_{s = \varepsilon \nu} \ \Lambda_g(s) \\
= H^0_f(\varepsilon \nu, w) \ \Res_{s = \varepsilon \nu} \ \Lambda_f(s) \cdot \im(z)^{\frac{1}{2} - \varepsilon \nu}
+ i \left( \frac{|z|}{z} \right) H^0_g(\varepsilon \nu, -w) \ \Res_{s = \varepsilon \nu} \ \Lambda_g(s) \cdot \im \left( -\frac{1}{Nz} \right)^{\frac{1}{2} - \varepsilon \nu}.
\end{multline}
Therefore
\begin{multline}\label{eq.tildefHtildegH}
    \tilde{f}(z) - H^0_f(\varepsilon \nu, w) \ \Res_{s = \varepsilon \nu} \ \Lambda_f(s) \cdot \im(z)^{\frac{1}{2} - \varepsilon \nu} \\
    = i \left( \frac{|z|}{z} \right) \cdot \left( \tilde{g} \left( -\frac{1}{Nz} \right) + H^0_g(\varepsilon \nu, -w) \ \Res_{s = \varepsilon \nu} \ \Lambda_g(s) \cdot \im \left( -\frac{1}{Nz} \right)^{\frac{1}{2} - \varepsilon \nu} \right).
\end{multline}
Since
\[
\hyp{ \frac{\varepsilon \nu + \nu + (1 + \varepsilon)/2}{2}}{\frac{\varepsilon \nu - \nu + (1 - \varepsilon)/2}{2} }{ \frac{1}{2} }{-w^2} = 1,\ \ \frac{\gamma_f^-(\varepsilon \nu + 1)}{\gamma_f^+(\varepsilon \nu)} = 0, 
\]
we know that $H^0_f(\varepsilon \nu, w) = H^0_g(\varepsilon \nu, -w) = 1$ and so equation~\eqref{eq.tildefHtildegH} becomes equation~\eqref{eq.gzf1z}. 

Analogously to \cite[Section 3.3]{TCTMF}, for a primitive Dirichlet character $\psi$ of conductor $q \in \mathcal{P}$ we define
\[
f_{\psi}(x + iy) = \sum_{n = 1}^{\infty} \frac{\psi(n)a_n}{\sqrt{\pi n}} \left( W_{\frac{1}{2}, \nu}(4 \pi ny) e(nx) + \psi(-1) \varepsilon \nu W_{-\frac{1}{2}, \nu}(4 \pi ny) e(-nx) \right),
\]
and similarly for $g_{\psi}$. Replacing $a_n$ by $\psi(n)a_n$ in equation~\eqref{eq.Mellinh} we know, for $\re(s)$ sufficiently large:
\begin{equation} \label{eq: mellin twisted f with lines}
H_f^{\frac{1 - \psi(-1)}{2}}(s, w) \Lambda_f(s, w) = (-i)^{\frac{1 - \psi(-1)}{2}} \int_0^{\infty} f_{\psi}(wy + iy) y^{s - \frac{1}{2}} \ \frac{dy}{y}.
\end{equation}
For $\ell \in \{ 0, 1 \}$, we have
\begin{equation}\label{eq.Hfidentityfull} 
H^{\ell}_f(1 - s, w) = i \left( \frac{|w + i|}{w + i} \right) (1 + w^2)^{s - \frac{1}{2}} H^{\ell}_f(s, -w).
\end{equation}
Given $a,b\in\mathbb{Z}$ such that if $c \in \{ a, b \}$ is non-zero then $(c, q) = 1$, set $\alpha = \frac{a}{q}$ and $\beta = \frac{b}{q}$. Since $\alpha,\beta\in\mathbb{R}$, we have
\[
f(z + \alpha) - f(z + \beta) = \tilde{f}(z + \alpha) - \tilde{f}(z + \beta) = \frac{1}{q - 1} \sum_{\psi~\mathrm{mod}~q} \tau(\overline{\psi}) (\psi(a) - \psi(b)) f_{\psi}(z). 
\]
Applying Mellin inversion to equation~\eqref{eq.Mellinh} gives

$$ f(z + \alpha) - f(z + \beta) = \sum_{\ell \in \{ 0, 1 \} } \frac{1}{2\pi i} \int_{(c)} \left( \Lambda_f\left(s, \alpha, \cos^{(\ell)}\right) - \Lambda_f\left(s, \beta, \cos^{(\ell)}\right) \right) H_f^{\ell}(s, w) y^{\frac{1}{2} - s} \ ds. $$
By equation~\eqref{eq.completedadditivetwist}, $\Lambda_f\left(s, \alpha, \cos^{(\ell)}\right) - \Lambda_f\left(s, \beta, \cos^{(\ell)}\right)$ is a linear combination of functions of the form $\Lambda_f(s, \psi)$ where $\psi$ is a primitive Dirichlet character of conductor $q$, and so $\Lambda_f(s, \alpha, \cos^{(\ell)}) - \Lambda_f(s, \beta, \cos^{(\ell)})$ is an entire function of $s$. Moreover each $H_f^{\ell}(s, w)$ is an entire function. Therefore we may use Cauchy's theorem to shift the path of integration to $\re(s) = 1 - c$ to obtain

\begin{multline}
    \int_{(c)} \left( \Lambda_f\left(s, \alpha, \cos^{(\ell)}\right) - \Lambda_f\left(s, \beta, \cos^{(\ell)}\right) \right) H_f^{\ell}(s, w) y^{\frac{1}{2} - s} \ ds \\
     = \int_{(1 - c)} \left( \Lambda_f\left(s, \alpha, \cos^{(\ell)}\right) - \Lambda_f\left(s, \beta, \cos^{(\ell)}\right) \right) H_f^{\ell}(s, w) y^{\frac{1}{2} - s} \ ds \\
     = \int_{(c)} \left( \Lambda_f(1 - s, \alpha, \cos^{(\ell)}) - \Lambda_f(1 - s, \beta, \cos^{(\ell)}) \right) H_f^{\ell}(1 - s, w) y^{s - \frac{1}{2}} \ ds.
\end{multline}
From equation~\eqref{eq.completedadditivetwists}, we deduce
\begin{multline*}
\Lambda_f\left(1 - s, \alpha, \cos^{(\ell)}\right) - \Lambda_f\left(1 - s, \beta, \cos^{(\ell)}\right) \\ = \frac{ i^{\ell} \chi(q) (q^2N)^{s - \frac{1}{2}} }{q - 1} \sum_{\substack{ \psi~\mathrm{mod}~q\\ \psi(-1) = (-1)^{\ell}}} \psi(-N) \tau(\psi) (\psi(a) - \psi(b)) \Lambda_g(s, \overline{\psi}).
\end{multline*}
Using this and equation \eqref{eq.Hfidentityfull} we obtain

\begin{multline*}
\int_{(c)} \left( \Lambda_f\left(s, \alpha, \cos^{(\ell)}\right) - \Lambda_f\left(s, \beta, \cos^{(\ell)}\right) \right) H_f^{\ell}(s, w) y^{\frac{1}{2} - s} \ ds \\
= \frac{ i^{\ell} \chi(q) }{q - 1} \cdot i \left( \frac{|w + i|}{w + i} \right)  \cdot \sum_{\substack{ \psi~\mathrm{mod}~q \\ \psi(-1) = (-1)^{\ell}}} \psi(-N) \tau(\psi) (\psi(a) - \psi(b)) \\ \cdot \int_{(c)} \Lambda_g(s, \overline{\psi}) H_g^{\ell} (s, -w) \left( \frac{1}{q^2N(1 + w^2)y} \right)^{\frac{1}{2} - s} \ ds.
\end{multline*}
Applying Mellin inversion to equation \eqref{eq: mellin twisted f with lines} gives

\begin{multline*}
\frac{1}{2 \pi i} \int_{(c)} \left( \Lambda_f(s, \alpha, \cos^{(\ell)}) - \Lambda_f(s, \beta, \cos^{(\ell)}) \right) H_f^{\ell}(s, w) y^{\frac{1}{2} - s} \ ds \\
= \frac{ i^{\ell} \chi(q) }{q - 1} \cdot i \left( \frac{|w + i|}{w + i} \right)  \cdot \sum_{\substack{ \psi~\mathrm{mod}~q \\ \psi(-1) = (-1)^{\ell}}} \psi(-N) \tau(\psi) (\psi(a) - \psi(b)) \cdot (-i)^{\ell} g_{\overline{\psi}} \left(-\frac{1}{q^2Nz} \right),
\end{multline*}
from which we deduce:
\[
f(z + \alpha) - f(z + \beta) = \frac{\chi(q)}{q - 1} \cdot i \left( \frac{|w + i|}{w + i} \right) \sum_{\psi~\mathrm{mod}~q} \psi(-N) \tau(\psi) ( \psi(a) - \psi(b) ) g_{\overline{\psi}} \left( -\frac{1}{q^2Nz} \right). 
\]
For all $a, b \not \equiv 0$ mod $q$ we conclude that:
\begin{multline*}
\sum_{\psi~\mathrm{mod}~q} \psi(a) \left( \tau(\overline{\psi}) f_{\psi} - i \chi(q) \psi(-N) \tau(\psi) g_{\overline{\psi}} \ |_1 W_{q^2N} \right) \\ = \sum_{\psi~\mathrm{mod}~q} \psi(b) \left( \tau(\overline{\psi}) f_{\psi} - i \chi(q) \psi(-N) \tau(\psi) g_{\overline{\psi}} \ |_1 W_{q^2N} \right).
\end{multline*}
The expression on the left hand side is therefore independent of $a \not \equiv 0$ mod $q$, so we have a linear combination of primitive characters producing the principal character modulo $q$. On the other hand, the set of Dirichlet characters mod $q$ is linearly independent over $\C$ so the coefficients of this linear combination must all vanish. Therefore, for any primitive Dirichlet character $\psi$ mod $q$: 
\begin{equation}\label{cor.gW}	
f_{\psi} = i \psi(-N) \chi(q) \frac{\tau(\psi)}{\tau(\overline{\psi})} g_{\overline{\psi}} \ |_1 W_{q^2N}. 
\end{equation}
Let $\psi$ be a primitive Dirichlet character of conductor $q \in \mathcal{P}$. We have
\[
f_{\psi} = \frac{1}{\tau(\overline{\psi})} \sum_{\substack{a \text{ mod } q \\ (a, q) = 1}} \overline{\psi}(a) f|_1 P_{\frac{a}{q}}.
\]
Using $f = i \ g \ |_1 H_N$ and equation~\eqref{cor.gW} we obtain
\begin{multline}
    -i \chi(q) \psi(-N) \tau(\psi) g_{\overline{\psi}}
    = \tau(\overline{\psi}) f_{\psi}  |_1 W_{q^2N} 
    = i \sum_{\substack{a \text{ mod } q \\ (a, q) = 1}} \overline{\psi}(a) g \ |_1 H_N \ P_{\frac{a}{q}} \ W_{q^2N} \\
    = i \sum_{\substack{a \text{ mod } q \\ (a, q) = 1}} \overline{\psi}(a) g \ |_1 \begin{pmatrix} -q^2N & 0 \\ qaN^2 & -N \end{pmatrix} 
    = -i \sum_{\substack{a \text{ mod } q \\ (a, q) = 1}} \overline{\psi}(a) g \ |_1 \begin{pmatrix} q^2 & 0 \\ -qaN & 1 \end{pmatrix} \\
    = -i \psi(-N) \sum_{\substack{a \text{ mod } q \\ (a, q) = 1}} \psi(a) g \ |_1 \begin{pmatrix} q & -a \\ - \tilde{a}N & \frac{Na \tilde{a} + 1}{q} \end{pmatrix} \ P_{\frac{a}{q}}.
\end{multline}
Here $\tilde{a}$ is an integer that is inverse to $-aN$ mod $q$; notice $\overline{\psi}(a) = \psi(-N) \psi(\tilde{a})$. Therefore for all primitive characters $\psi$ we have
\begin{equation}\label{eq.identity}
\chi(q) \sum_{\substack{a \text{ mod } q \\ (a, q) = 1}} \psi(a) g \ |_1 P_{\frac{a}{q}} = \sum_{\substack{a \text{ mod } q \\ (a, q) = 1}} \psi(a) g \ |_1 \begin{pmatrix} q & -a \\ - \tilde{a}N & \frac{Na \tilde{a} + 1}{q} \end{pmatrix} \ P_{\frac{a}{q}}.
\end{equation}
The argument from now on is very similar to \cite[Section~3.3]{TCTMF} and \cite[Section~1.5]{AFAR}. As the primitive Dirichlet characters of conductor $q$ span the complex vector space 
\begin{equation}\label{eq.PrimSpan}
V = \left\{ \theta: (\Z / q\Z)^{\times} \rightarrow \C : \sum_{\substack{a \text{ mod } q \\ (a, q) = 1}} \theta(a) = 0 \right\},
\end{equation}
we can replace $\psi$ in equation~\eqref{eq.identity} with any function $\theta\in V$. Take $s \in \mathcal{P} \ \backslash \ \{q\}$ and write $qs = 1 + r \tilde{r} N$ for some integers $r, \ \tilde{r}$. Define $\theta_1\in V$ by
\begin{align*}
    \theta_1(n) = \begin{cases} \pm 1, & n \equiv \pm r \text{ mod } q, \\ 0 & \text{otherwise}. \end{cases}
\end{align*}
Equation~\eqref{eq.identity} then becomes
\begin{align*}
    \chi(q) \left( g \ |_1 P_{\frac{r}{q}} - g \ |_1 P_{-\frac{r}{q}} \right) = g \ |_1 \begin{pmatrix} q & -r \\ - \tilde{r}N & s \end{pmatrix} \ P_{\frac{r}{q}} - g \ |_1 \begin{pmatrix} q & r \\  \tilde{r}N & s \end{pmatrix} \ P_{-\frac{r}{q}}.
\end{align*}
We define
\[
A_{\pm} = \begin{pmatrix} q & \pm r \\ \pm \tilde{r}N & s \end{pmatrix},
\]
which satisfies
\[
A_{\pm}^{-1} = \begin{pmatrix} s & \mp r \\ \mp \tilde{r}N & q \end{pmatrix}.
\] 
We have
\begin{equation*}
    (A_+ - \chi(q)) P_{-\frac{r}{q}} \equiv (A_- - \chi(q)) P_{\frac{r}{q}},\ \
    (A_-^{-1} - \chi(s)) P_{-\frac{r}{s}} \equiv (A_+^{-1} - \chi(s)) P_{\frac{r}{s}},
\end{equation*}
and so $(A_+ - \chi(q))(1 - M(q, s, r)) \in \Omega_1$ where
\[
M(q, s, r) = A_+^{-1} P_{\frac{2r}{s}} A_- P_{\frac{2r}{q}} = \begin{pmatrix} 1 & \frac{2r}{q} \\ -\frac{2\tilde{r}N}{s} & -3 + \frac{4}{qs} \end{pmatrix} \in \text{SL}_2(\R).
\]
The matrix $M(q, s, r)\in\SL_2(\R)$ is an elliptic matrix of infinite order, since its eigenvalues are not roots of unity and the absolute value of its trace is $<2$. Let $\gamma = \begin{pmatrix} a & b \\ cN & d \end{pmatrix}$ be an arbitrary element of $\Gamma_0(N)$. Choose distinct primes $q, \ s \in \mathcal{P}$ with $q = a - ucN$ and $s = d - vcN$ for integers $u$ and $v$. Let $r = b - av + uvcN - ud$, so that
\[
g |_1 \begin{pmatrix} a & b \\ cN & d \end{pmatrix} = g |_1 P_u \begin{pmatrix} q & r \\ cN & s \end{pmatrix} P_v = g |_1 \begin{pmatrix} q & r \\ cN & s \end{pmatrix} P_v.
\]
We have $G_1 \ |_1 \ M(q, s, r) = G_1$ where $G_1 = g |_1 \begin{pmatrix} q & r \\ cN & s \end{pmatrix} - \chi(a) g$ (note $\chi(q) = \chi(a)$). By Lemma \ref{prop: one circle} we must have $G_1 = 0$, which establishes the weight-$1$ modular transformation law. To complete the proof of the proof of Theorem~\ref{thm:WeilMaass}, we note that the correct growth of $f(z)$ and $g(z)$ at the cusps of $\Gamma_0(N)$ is a consequence of the assumption that $L_f(s)$ and $L_g(s)$ converge in some right half-plane\footnote{The proof given in \cite[V-15]{Ogg} can be modified to real-analytic forms.}.  
\subsection{Proof of Theorem~\ref{thm.DPZmaass}}\label{TwistingModuli}
Let $k\in\{0,1\}$ and say $q$ is a prime such that $q \nmid N$. For all primitive $\psi$ mod $q$, we have
\begin{equation} \label{eq: no prop 3 8}
 \sum_{\substack{a \text{ mod } q \\ (a, q) = 1}} \psi(a) f |_k P_{\frac{a}{q}} = \sum_{\substack{a \text{ mod } q \\ (a, q) = 1}} \psi(a) f |_k \begin{pmatrix} q & -a \\ - \tilde{a}N & \frac{Na \tilde{a} + 1}{q} \end{pmatrix} \ P_{\frac{a}{q}}.
\end{equation}
Indeed, when $k=0$ this follows from \cite[equation~(3.16)]{TCTMF} and when $k=1$ this is equation~\eqref{eq.identity}. Though the argument in \cite{DPZ} is written for a group action slightly different to \eqref{eq.weightkaction}, analysing the proof we see that it may be modified to our situation using equations \eqref{eq.periodic} and \eqref{eq.FEslash} and the axiomatic properties of group actions. 
\subsection{Proof of Theorem~\ref{thm.CFMaass}}\label{sec.CFmaass}
The Fourier expansion of $f$ immediately implies it is an eigenfunction of $\Delta_k$. Below we will establish the modular transformation laws with respect to $\Gamma_0(N)$. Given that, the correct growth property of $f(z)$ at the cusps of $\Gamma_0(N)$ is a consequence of the convergence of $L_f(s)$ in some right half-plane. 
\begin{proof}[Proof when $1\leq N\leq4$]
In this case $\Gamma_0(N)$ is generated by the matrices $P_1,H_NP_{-1}H_N^{-1}$, and so the result follows from equations~\eqref{eq.periodic} and \eqref{eq.FEslash}.
\end{proof}
\begin{proof}[Proof when $N\in\{5,6,7,8,9,10,12,16\}$]
The proof given in \cite{EOHCT} does not require any elliptic matrices and so can be modified to real-analytic setting using the action~\eqref{eq.weightkaction} and equations~\eqref{eq.periodic},~\eqref{eq.FEslash}, and \eqref{eq.apTp}.
\end{proof}
For $N\in\mathbb{Z}$, introduce
\[
Q=\begin{pmatrix}
-1&0\\0&-1
\end{pmatrix},\ \ W_N=\begin{pmatrix}
1&0\\ N &1
\end{pmatrix}.
\]
In particular,
\begin{equation}\label{eq.equiv1}
Q\equiv1,\ \ W_N=H_NP_{-1}H_N^{-1}\equiv1.
\end{equation}
\begin{proof}[Proof when $N=11$]
	When $k=1$ the result follows from \cite{EOHCT} and Lemma~\ref{prop: one circle}. When $k=0$, analysing the proof given in \cite{EOHCT} it suffices to show that $f$ is invariant under the matrix $\begin{pmatrix}3&-1\\-11&4\end{pmatrix}$. From [loc. cit.], we know that
	\[1-\begin{pmatrix}3&-1\\-11&4\end{pmatrix}=\left(1-\begin{pmatrix}3&-1\\-11&4\end{pmatrix}\right)\begin{pmatrix}1&-2/3\\11/2&-8/3 \end{pmatrix},\]
	in which the final matrix on the right-hand side is elliptic of infinite order. On the other hand, we know that
	\[W_{11}=\begin{pmatrix}2&-1\\11&5\end{pmatrix}P\begin{pmatrix}3&-1\\-11&4\end{pmatrix}\begin{pmatrix}2&-1\\11&5\end{pmatrix}^{-1}\begin{pmatrix}3&-1\\-11&4\end{pmatrix}^{-1}\equiv1,\]
	and that $\begin{pmatrix}2&-1\\11&5\end{pmatrix}\equiv1$ by [loc. cit.]. Combining these two facts, we deduce that
	\[1-\begin{pmatrix}3&-1\\-11&4\end{pmatrix}=\left(1-\begin{pmatrix}3&-1\\-11&4\end{pmatrix}\right)\begin{pmatrix}2&-1\\11&5\end{pmatrix}^{-1}.\]
	The final matrix on the right-hand side does not have the same fixed point as $\begin{pmatrix}1&-2/3\\11/2&-8/3 \end{pmatrix}$ and so we are done by Lemma~\ref{prop: one circle}.
\end{proof}
We introduce
\[
A=\begin{pmatrix} 7&-1\\36&-5\end{pmatrix}\in\Gamma_0(18),\ \ B=\begin{pmatrix} 13&-8\\18&-11\end{pmatrix}\in\Gamma_0(18),
\]
\[
C=\begin{pmatrix} 13&-2\\20&-3\end{pmatrix}\in\Gamma_0(20),\ \ D=\begin{pmatrix} 19&-4\\24&-5\end{pmatrix}\in\Gamma_0(24).
\]
One may check that the table below records a set of generators for $\Gamma_0(N)$ when $N\in\{18,20,24\}$.
\begin{center}
	\begin{tabular}{|c||c|}\hline
		$N$ & \text{Generators for }$\Gamma_0(N)$\\
		\hline\hline
		18&$ P_1,\ \ Q,\ \ A,\ \ B,\ \ \begin{pmatrix} 71 & -15 \\ 90 &-9 \end{pmatrix},\ \ 
		\begin{pmatrix} 55 & -13 \\ 72 &-17 \end{pmatrix},\ \ 
		\begin{pmatrix} 7 & -2\\ 18 &-5 \end{pmatrix}, \ \ 
		\begin{pmatrix} 31 & -25 \\ 36 &-29 \end{pmatrix}.$\\
		\hline
		20&$\begin{array}{c} P_1,\ \ Q,\ \ C,\ \ \begin{pmatrix} 49 & -9 \\ 60 &-11 \end{pmatrix},\ \ \begin{pmatrix} 31 & -7 \\ 40 &-9 \end{pmatrix},\ \ \begin{pmatrix} 29 & -8 \\ 40 &-11 \end{pmatrix},\\ \begin{pmatrix} 31 & -9 \\ 100 &-29 \end{pmatrix},\ \ \begin{pmatrix} 17 & -6 \\ 20 &-7 \end{pmatrix}.\end{array}$\\
		\hline
		24& $\begin{array}{c} P_1,\ \ Q,\ \ D,  \ \ \begin{pmatrix} 
		19 & -2 \\ 48 &-5 
		\end{pmatrix},\ \ 
		\begin{pmatrix} 
		61 & -7 \\ 
		96 &-11 
		\end{pmatrix},\ \
		\begin{pmatrix} 
				59 & -8 \\ 
				96 &-13
				\end{pmatrix},\ \ 
				\begin{pmatrix} 
						13 & -2 \\ 
						72 &-11 
						\end{pmatrix},\\
			\begin{pmatrix} 
					17 & -5 \\ 
					24 &-7 
					\end{pmatrix},\ \
					\begin{pmatrix} 
							61 & -25 \\ 
							144 &-5 
							\end{pmatrix},\ \
							\begin{pmatrix} 
									13 & -6 \\ 
									24 &-11 
									\end{pmatrix},\ \
									\begin{pmatrix} 
											-5 & -2 \\ 
											48 &-19 
											\end{pmatrix}. \end{array}$\\
		\hline
	\end{tabular}
\end{center}
For $N\equiv0$ mod $2$, define
\[
J_N=\frac{1}{\sqrt{2}}\begin{pmatrix}
-2 & 1 \\ N & -\frac{N+2}{2}
\end{pmatrix}\in\SL_2(\R).
\]
Though the argument in \cite[Lemma~3]{EOHCT} is written for a group action slightly different to \eqref{eq.weightkaction}, analysing the proof we see that it may be modified to show that $J_{18}\equiv-1$ provided $2^2$ does not divide $N$. On the other hand, for $N\equiv0$ mod 4 define
\[
L_N=\begin{pmatrix}
\frac{N}{4}-1 & -\frac12 \\ \frac{N}{2} & -1
\end{pmatrix}\in\SL_2(\R),
\]
which satisfies $L_N\equiv NL_N=(P_{\frac12}H_N)^2\equiv1$. We may now rewrite the table above.
\begin{center}
	\begin{tabular}{|c||c|}\hline
		$N$ & \text{Generators for }$\Gamma_0(N)$\\
		\hline\hline
		18&$\begin{array}{c}P_1,\ \ Q,\ \ A,\ \ B,\ \ BH_{18}J_{18}W_{18}J_{18}H_{18}^{-1}W_{18}^{-1}A^{-1},\\ QBJ_{18}^{-2}W_{18}^{-1}A^{-1},\ \ Q(W_{18}J_{18})^{-2}A^{-1}, \ \ QP_1H_{18}J_{18}^{-1}W_{18}^{-1}J_{18}^{-1}H_{18}^{-1}B^{-1}. \end{array}$\\
		\hline
		20&$\begin{array}{c}P,\ \ Q, \ \ C \ \ CL_{20}^{-2}W_{20}^{-1}L_{20}^{-1}CQ,\ \ CL_{20}^{-2}CL_{20}^{-1},\\ 
		CL_{20}^{-1}W_{20}^{-1}L_{20}^{-1}CL_{20}^{-1}Q, \ \ L_{20}^2C^{-1}W_{20}^{-1}L_{20}^{-1}CL_{20}^{-1},\ \ CL_{20}^{-3}.\end{array}$\\
		\hline
		24& $\begin{array}{c} P_1,\ \ Q,\ \ D,  \ \ L_{24}^2,\ \ L_{24}W_{24}P_1^{-1}DL_{24}, \ \ QP_{1/2}L_{24}^{-1}D^{-1}P_{1/2}L_{24}\\ H_{24}P_1^{-1}DP_{1/2}^{-1}DP_{1/2}H_{24}^{-1}, \ \ QP_{1/2}D^{-1}P_{1/2} \\ L_{24}P_1^{-1}L_{24}^{-1}, \ \ L_{24}W_{24}L_{24}^{-1}, \ \ (W_{24}^{-1}L_{24}^{-1}P_1)^2. \end{array}$\\
		\hline
	\end{tabular}
\end{center}
Note that these generating sets are far from minimal, for example, the list in the case $N=20$ may be simplified to:
\[P_1,\ \ Q,\ \ C,\ \ L_{20}W_{20}L_{20}^{-1},\ \ L_{20}^{-1}W_{20}L_{20},\ \ L_{20}CL_{20}^{-1},\ \ L_{20}^{-1}CL_{20},\ \ L_{20}^3.\]
\begin{proof}[Proof in case $N=18$]
Since $4\nmid 18$, we have that $J_{18}\equiv1$. By the list of generators for $\Gamma_0(18)$ given in the table and equations \eqref{eq.periodic}, \eqref{eq.FEslash}, and \eqref{eq.equiv1}, it suffices to prove that $\gamma\equiv1$ for $\gamma\in\{A,B\}$. Noting that
\[
B=PH_{18}AH_{18}^{-1}P^{-1},\ \ A\equiv\begin{pmatrix}-11&-1\\-54&-5 \end{pmatrix}W_{18}^{-1},
\]
it is enough to show that $\begin{pmatrix}-11&-1\\-54&-5\end{pmatrix}\equiv1$. Because $3^2|N$, we have $f|_kU_3=0$ which implies
$P_{1/3}+P_{-1/3}\equiv-1$. Therefore:
\begin{equation}\label{eq.P13P-13}
(P_{1/3}+P_{-1/3})J_{18}\equiv1\equiv J_{18}(P_{1/3}+P_{-1/3}).
\end{equation}
Multiplying equation~\eqref{eq.P13P-13} by $P_{1/3}J_{18}$, we deduce:
\[
\begin{pmatrix}-11&-1\\-54&-5 \end{pmatrix}+\begin{pmatrix}25&-7/3\\-54&-5 \end{pmatrix}\equiv\begin{pmatrix}13&4/3\\-108&-11 \end{pmatrix}+1,
\]
and on the other we have
\[
\begin{pmatrix}13&4/3\\-108&-11 \end{pmatrix}\equiv\begin{pmatrix}25&-7/3\\-54&-5 \end{pmatrix}.
\]
Combining the previous two equations, we are done.
\end{proof}
\begin{proof}[Proof in case $k=1$ and $N\in\{14,15,17,20,23,24\}$]
For $N\in\{14,15,17,23\}$, we may mimic the argument given in \cite{EOHCT} using Lemma~\ref{prop: one circle} in place of \cite[Lemma~5]{EOHCT}. When $N=20$, it suffices to prove that $C\equiv1$. We have
\begin{equation}\label{eq.P13}
\left(1-\begin{pmatrix} 3 & -1 \\ 40 & -13 \end{pmatrix}P_{1/3}\right)+\left(1-\begin{pmatrix} 3 & -2 \\ 20 & -13 \end{pmatrix}P_{2/3}\right)\equiv0.
\end{equation}
Noting that
\[
\begin{pmatrix} 3 & -1 \\ 40 & -13 \end{pmatrix}\equiv CL_{20}^{-2},\ \ \begin{pmatrix} 3 & -2 \\ 20 & -13 \end{pmatrix}\equiv C^{-1},
\]
equation~\eqref{eq.P13} becomes
\[
1-C\equiv(1-C)C^{-1}P_{1/3}L^2_{20}.
\]
The matrix $C^{-1}P_{1/3}L^2_{20}$ is elliptic of infinite order and so the result follows from Lemma~\eqref{prop: one circle}. When $N=24$, it suffices to prove that $D\equiv1$. We have
\begin{multline}\label{eq.P15}
\left(1-\begin{pmatrix} 5 & -4 \\ 24 & -19 \end{pmatrix}P_{4/5}\right)+\left(1-\begin{pmatrix} 5 & 4 \\ -24 & -19 \end{pmatrix}P_{-4/5}\right)\\
+\left(1-\begin{pmatrix} 5 & -2 \\ 48 & -19 \end{pmatrix}P_{2/5}\right)+\left(1-\begin{pmatrix} 5 & 2 \\ -48 & -19 \end{pmatrix}P_{-2/5}\right)\equiv0.
\end{multline}
Noting that
\[
\begin{pmatrix} 5 & -4 \\ 24 & -19 \end{pmatrix}\equiv QD^{-1}\ \ \begin{pmatrix} 5 & 4 \\ -24 & -19 \end{pmatrix}\equiv DP,
\]
\[
\begin{pmatrix} 5 & -2 \\ 48 & -19 \end{pmatrix}\equiv QL_{24}^{-2}\equiv1,\ \ \begin{pmatrix} 5 & 2 \\ -48 & -19 \end{pmatrix}\equiv Q(W_{24}^{-1}L_{24}^{-1}P)^2\equiv1,
\]
equation \eqref{eq.P15} becomes
\[
1-D\equiv (1-D)D^{-1}P_{3/5}.
\]
The matrix $D^{-1}P_{3/5}$ is elliptic of infinite order. The result now follows from Lemma~\ref{prop: one circle}.
\end{proof}
\subsection{Proof of Theorem~\ref{thm.Holomorphic}}
Given a function $f$ on $\mathcal{H}$ and $\gamma=\begin{pmatrix} a&b\\c&d\end{pmatrix}\in\SL_2(\R)$, define the function $f|^k\gamma$ by
\[
\left(f|^k\gamma\right)(z)=\det(\gamma)^{k/2}(cz+d)^{-k}f\left(\frac{az+b}{cz+d}\right).
\]
Let $f$ be defined as in \eqref{eq.holexp}. By equation~\eqref{eq.holexp} (resp. \eqref{eq.holFE}), we know that
\[
f=f|^kP,\ \ f=f|^kH_N.
\]
We may now argue as in Section~\ref{sec.CFmaass}, replacing Lemmas~\ref{prop:2circles} and~\ref{prop: one circle} by:
\begin{lemma}[Lemma~5 in \cite{EOHCT}]\label{lem.holWeil}
Suppose $F$ is holomorphic on $\uhp$ and $E\in\SL_2(\R)$ is elliptic. If $f|^kE=f$, then either $E$ has finite order or $f$ is constant.
\end{lemma}
Holomorphy of $f$ at the cusps follows from convergence of $L_f(s)$ in some right half-plane, as in \cite[V-14]{Ogg}.
\bibliography{TwistingModuli}
\bibliographystyle{alpha}
\end{document}